\documentclass{amsart}
\usepackage[utf8]{inputenc}
\usepackage{indentfirst}
\usepackage{setspace}
\usepackage{amsmath,amsfonts,amsthm}
\usepackage{amssymb}
\everymath{\displaystyle}
\usepackage{fancyhdr}
\usepackage{graphicx}
\usepackage{enumerate}
\usepackage{xcolor}
\usepackage[backref]{hyperref}

\newcommand{\F}{\mathbb{F}}

\newcommand{\CM}{\mathcal{M}}

\newcommand{\bfA}{\mathbf{A}}
\newcommand{\bfB}{\mathbf{B}}
\newcommand{\bfC}{\mathbf{C}}
\newcommand{\bfD}{\mathbf{D}}
\newcommand{\bfE}{\mathbf{E}}

\newcommand{\bfU}{\mathbf{U}}

\newcommand{\Tr}{\mathrm{Tr}}
\newcommand{\diag}{\mathrm{diag}}
\newcommand{\Diag}{\mathrm{Diag}}
\newcommand{\ord}{\mathrm{ord}}

\theoremstyle{plain}
\newtheorem*{thm*}{Theorem}
\newtheorem{thm}{Theorem}[section]
\newtheorem{lemma}[thm]{Lemma}
\newtheorem{prop}[thm]{Proposition}
\newtheorem{cor}[thm]{Corollary}

\newtheorem*{thm-intro}{Theorem}

\theoremstyle{definition}

\theoremstyle{remark}
\newtheorem{rem}[thm]{Remark}

\title{On the Waring Problem for Matrices over Finite Fields}
\author{Simion Breaz }
\date{}

\begin{document}

\begin{abstract}
We prove that if $k$ is a positive integer, then for every finite field $\F_q$ of cardinality $q\neq 2$ and for every positive integer $n$ such that $q^n>(k-1)^4$, every $n\times n$ matrix over $\F_q$ can be expressed as a sum of two $k$-th powers.     
\end{abstract}



\address{Babe\c s-Bolyai University, Faculty of Mathematics and Computer Science, Str. Mihail Kog\u alniceanu 1, 400084, Cluj-Napoca, Romania}
\email[Simion Breaz]{simion.breaz@ubbcluj.ro; bodo@math.ubbcluj.ro}

\maketitle

\section{Introduction}

Waring's Problem for matrices asks, for a given positive integer $k$, for decompositions of square matrices over a fixed ring into sums of $k$-th powers. Such decompositions are called Waring decompositions, and they have been investigated in many papers. For instance, it was proved in \cite{Ri} that if all $n\times n$ matrices are sums of $k$-th powers, $k\leq n$, then they are sums of seven $k$-th powers. These results were extended in \cite{KG}, where the existence of Waring decompositions for a matrix is characterized using the trace. A detailed study of Waring decompositions for finite rings, in particular for matrix rings over finite fields, is presented in \cite{DK}. 

For matrices over finite fields $\F_q$ (of cardinality $q$), it was observed in \cite{KiS23} that the problem of finding minimal conditions for $q$ and $k$ such that every matrix is a sum of three or two $k$-th powers corresponds to a similar question, solved in \cite{S} and \cite{LST}, for non-commutative simple groups.
Such minimal conditions are also described in \cite{Ki22}, \cite{KiS23} and \cite{KVZ23}. For instance, it was proved in \cite[Theorem 6.1]{KiS23} that for every $k$ there exists a constant $C_k\leq k^{16}$ depending only on $k$ such that if $q>C_k$ then every $n\times n$ matrix ($n\geq 2$ in \cite{KVZ23}) is a sum of two $k$-th powers. For the case $\gcd(k,q)=1$ this was improved in \cite[Corollary 1.6]{KVZ23}, where it is proved that we can take $C_k=k^3-3k^2+3k$. These results are connected to a conjecture of Larsen, \cite[Problem 1]{KVZ23}, which states that for every $k$ there exists a constant $C_k$ such that if $q^{n^2}>C_k$, then every $n\times n$ matrix is a sum of two $k$-th powers. Unfortunately, the above-mentioned results yield only lower bounds for $q$ given by polynomial expressions in $k$, involving radicals, and do not capture bounds of the form $q^{f(n)}>C_k$, where $f$ is an increasing function (for instance, affine or quadratic).   

We will prove in Theorem \ref{thm:nonscalar} that if $q\neq 2$ and we have the inequality 
\begin{equation}
\label{eq:d<general}k<\frac{q^{n/2}-\omega(n)q+q-1}{q-1} ,
\end{equation} then every non-scalar matrix is a sum of two $k$-th powers ($\omega(n)$ represents the number of prime divisors of $n$). 
In particular, we obtain in Corollary \ref{cor:main} a weak version of Larsen's conjecture: if $q^n>(k-1)^4$ then every $n\times n$ matrix over a finite field $\F_q\neq \F_2$ is a sum of two $k$-th powers. The case of matrices over $\F_2$ was investigated in \cite{KVZ23}. For instance, it is proved in \cite[Theorem 5.1]{KVZ23} that for $k$ odd, all matrices over $\F_2$ are sums of two $k$-th powers.   

In this paper $\F_q$ will be a field with $q$ elements, and the references used here for Finite Field Theory are \cite{lidl} and \cite{Sch76}. The ring of all $n\times n$-matrices over a field $\F$ is denoted by $\CM_n(\F)$. If $P=X^n-\alpha_{n-1}X^{n-1}-\dots-\alpha_1X-\alpha_0\in \F[X]$ then we call $\alpha_{n-1}$ the \textit{trace} of $P$, and we denote it by $\Tr(P)$. We also denote by $\Tr(A)$ the trace of every square matrix.

{
For every monic polynomial $P=X^n-\alpha_{n-1}X^{n-1}-\dots-\alpha_1X-\alpha_0\in \F[X],$ there exists a \textit{companion matrix}
$$\bfC_P=\bfC_{\alpha_0,\dots,\alpha_{n-1}} = \begin{pmatrix}
0 & 0 & \cdots & 0 & \alpha_{0} \\
1 & 0 & \cdots & 0 & \alpha_{1} \\
\vdots  & \vdots  & \ddots & \vdots & \vdots  \\
0 & 0 & \cdots & 0 & \alpha_{n-2} \\
0 & 0 & \cdots & 1 & \alpha_{n-1}
\end{pmatrix},
$$
such that the characteristic polynomial of $\bfC$ is $P$. Note that the minimal polynomial of $\bfC$ is also $P$. A matrix $\bfA\in \CM_n(\F)$ is called \textit{non-derogatory} if the degree of its minimal polynomial is $n$. 
As a reminder, every matrix over a field $\F$ is similar to a \textit{Frobenius canonical form}, which is a unique block-diagonal matrix \(\text{diag}(\bfC_{P_1}, \bfC_{P_2}, \dots, \bfC_{P_s})\), where each \(\bfC_{P_i}\) is the companion matrix of a monic polynomial $P_i$ such that $P_i \mid P_{i+1}$ for all $1 \leq i < s$. These polynomials are the invariant factors of the initial matrix. In particular, from the Frobenius canonical form, it follows that every non-derogatory matrix is similar to a companion matrix. Moreover, two non-derogatory matrices are similar if and only if they have the same characteristic polynomial. 

\section{Waring decomposition for matrices}

{	
We fix a finite field $\F_q$ and a positive integer $n\geq 2$. We will work in the finite field extension $\F_{q^n}/\F_q$. Let $\phi:\F_{q^n}\to \F_{q^n}$, $\phi(x)=x^q$ the Frobenius automorphism associated to this extension. For every $x\in \F_{q^n}$ we consider the tuple of its \textit{conjugates} $\mathbf{t}=(x, \phi(x), \dots, \phi^{n-1}(x))$ with respect to this field extension. We will say that $\mathbf{t}$ is the \textit{orbit} of $x$. We will also use the polynomial $$P_{\mathbf{t}}=P_x=(X-x)(X-\phi(x))\dots(X-\phi^{n-1}(x))\in \mathbb{F}_q,$$ induced by $\mathbf{t}$. We will denote by $$\Tr_{q^n/q}:\F_{q^n}\to \F_q, \ \Tr_{q^n/q}(x)=x+\phi(x)+ \dots+ \phi^{n-1}(x),$$ the trace of $x$ with respect to the extension $\F_{q^n}/\F_q$. We remind the reader that $\Tr_{q^n/q}(x)$ coincides with the trace of the associated polynomial.

\begin{rem}\label{rem:tuple-si-urma}
Using \cite[Theorem 2.14]{lidl} and the remark after \cite[Definition 2.17]{lidl}, it follows that $f_{\mathbf{t}}$ is irreducible if and only if the orbit $\mathbf{t}$ has no repetitions. 
Moreover, every monic irreducible polynomial is of the form $P_{\mathbf{t}}$, where $\mathbf{t}$ is an orbit without repetitions.  
\end{rem}

In order to construct Waring decompositions for matrices over $\F_q$, we need to identify polynomials $P$ such that the associated companion matrices $\bfC_{P}$ are $k$-th powers (we refer to  \cite{KS-25} for more details about such polynomials). We will apply the following 

\begin{lemma}\label{lem:k-power}
Let $k$ be a positive integer. If $a\in \F_{q^n}$ and the values $$a^k,\phi(a^k),\dots,\phi^{n-1}(a^k)$$ are pairwise distinct (i.e., the polynomial $P_{a^k}$ is irreducible), then the companion matrix $\bfC_{P_{a^k}}$ is a $k$-th power in $\CM_n(\F)$.    
\end{lemma}

\begin{proof}
According to our hypothesis, $\bfC_{P_{a}}$ and $\bfC_{P_{a^k}}$ are diagonalizable in $\CM_{n}(\F_{q^n})$. If $\bfU\in \mathrm{GL}_n(\F_{q^n})$ is a matrix such that $\bfU^{-1}\Diag(a,\phi(a),\dots,\phi^{n-1}(a))\bfU=\bfC_{P_a}$ then $$\bfU^{-1}\Diag(a^k,\phi(a^k),\dots,\phi^{n-1}(a^k))\bfU=(\bfC_{P_a})^k.$$ It follows that the matrix $(\bfC_{P_a})^k$ is diagonalizable in $\CM_n(\F_{q^n})$ and its eigenvalues are pairwise distinct. Using the Frobenius canonical form, it follows that $(\bfC_{P_a})^k$ is similar to the companion matrix associated to the polynomial $P_{a^k}$, which completes the proof. 
\end{proof}

\begin{cor}\label{cor:power-k-companion}
Suppose that $b\in \F_{q^n}$ is a primitive element. Then, for every positive integer $k< q^{\frac{n}{2}}+1$ the polynomial $P_{b^k}$ is irreducible, and the companion matrix $\bfC_{P_{b^k}}$ is a $k$-th power  in $\CM_n(\F)$.
\end{cor}

\begin{proof}
Suppose that $P_{b^k}$ is not irreducible. Therefore, $b^k,\phi(b^k),\dots,\phi^{n-1}(b^k)$ are not pairwise distinct. It follows that $b^k$ belongs to a proper subfield of $\F_{q^n}$. Since $b^k\notin \F_q$, there exists an integer $1< u\leq \lfloor \frac{n}{2}\rfloor$ such that $b^{k(q^u-1)}=1$ and $u\mid n$. 
Since $b$ is primitive, we have $q^n-1\mid k(q^u-1)$. Using the hypothesis $k< q^{\frac{n}{2}}+1$ we obtain a contradiction. 

For the last statement, we apply Lemma \ref{lem:k-power} to conclude that the matrix $\bfC_{P_{b^k}}$ is a $k$-th power.  \end{proof}

In the following result, we will prove that under a suitable hypothesis there exist irreducible polynomials as in Lemma \ref{lem:k-power} with arbitrary traces. The existence (up to some trivial cases) of primitive polynomials with arbitrary traces was proved in \cite{Coh}. For a survey on results concerning the existence of irreducible polynomials with prescribed coefficients, we refer to \cite{Coh2}. 
    
\begin{prop}\label{prop:k-power-polynomial}
If $n$ and $k$ are positive integers with $\gcd(k,q)=1$ and 
\begin{equation*}
	\tag{1}		k<\frac{q^{n/2}-\omega(n)q+q-1}{q-1} \end{equation*}
then for every $a\in \F_q$ there exists an irreducible polynomial $P$ of degree $n$ such that the companion matrix $\bfC_P$ is a $k$-th power and $\Tr(P)=a$.
	\end{prop}
	
\begin{proof} 
We denote by $M_q(n)$ the cardinality of the set $$M=\{x\in \F_{q^n}\mid \textrm{the orbit of  } x \textrm{ has no repetitions}\}.$$
Let us observe that two orbit tuples without repetitions, $t_1$ and $t_2$, induce the same polynomial if and only if $t_2$ is obtained as a cyclic permutation of $t_1$. It follows that $M_q(n)=nN_q(n)$, where $N_q(n)$ is the number of monic irreducible polynomials of degree $n$ over $\mathbb{F}_q$. We apply \cite[Theorem 3.25]{lidl} to obtain 
	\begin{equation}\label{eq:nr-tuple}
		M_q(n) =\sum_{d \mid n} \mu\left(d\right) q^{\frac{n}{d}}\geq q^n - \omega(n) \cdot q^{{\lfloor n/2 \rfloor}},
	\end{equation} 
where $\mu(-)$ represents the M\"obius function, and $\omega(n)$ is the number of prime divisors of $n$.

For every $a\in \F_q$, let $N_a$ be the number of solutions for the equation $\Tr_{q^n/q}(x^k)=a$.
By Weil Theorem, \cite[Theorem 6.61]{lidl}, it follows that \begin{equation} \label{eq:N_a} \left| N_a - q^{n-1} \right|\leq \frac{1}{q} \cdot (q-1) \cdot (k-1)\sqrt{q^n} \end{equation}

Using the discussion presented in Remark \ref{rem:tuple-si-urma}, we observe that, in order to conclude that there exists a polynomial $P$ such that $\Tr(P)=a$ and $\bfC_P$ is a $k$-th power, it is enough to verify if $N_a+M_q(n)>q^n$. Under the hypothesis \eqref{eq:d<general}, this follows from \eqref{eq:nr-tuple} and \eqref{eq:N_a}.
	\end{proof}

In order to prove our main result about Waring decompositions, we will use some additive decompositions for matrices. 
Recall that a matrix is an \textit{unreduced upper Hessenberg matrix} if for every $j\in \{1,\dots,n-1\}$ we have $a_{(j+1)j}\neq 0$ and for every $i>j+1$ we have $a_{ij}=0$. For the reader's convenience, some details about the proof of the following well-known lemma are added. The second statement says that every non-derogatory matrix is similar to an unreduced Hessenberg matrix whose first $(n-1)$ columns are prescribed.

\begin{lemma}\label{lem:non-derogatory}
Let $\F$ be a field and $n$ a positive integer. 

{\rm (i)} If $\bfA=(a_{ij})\in \CM_n(\F)$ is an unreduced upper Hessenberg matrix then it is non-derogatory.

{\rm (ii)} Suppose that $\bfB\in \CM_{n}(\F)$ is a non-derogatory matrix and that we have fixed some elements $a_{ij}\in \F$ with $1\leq i\leq n$ and $1\leq j\leq n-1$ such that $a_{(j+1)j}\neq 0$ for all $j$, and for every $i>j+1$ we have $a_{ij}=0$. Then there exist $a_{1n},\dots,a_{nn}\in \F$ such that $\bfB$ is similar to the matrix $\bfA=(a_{ij})$.    
\end{lemma}

\begin{proof}
(i) If $e_1=(1,0,\dots,0)^t$ and $c_j(\bfA)$ represents the $j$-th column of $\bfA$ then $(e_1,c_1(\bfA),\dots,c_{n-1}(\bfA))$ is a basis for $\F^n$. Using this basis, we observe that $\bfA$ is similar to a companion matrix.

(ii) For simplicity, we can assume that $\bfB$ is a companion matrix. It follows that the vectors 
$$f_1=e_1, f_2=a_{21}^{-1}(Bf_1-a_{11}f_1), f_3=a_{32}^{-1}(Bf_2-a_{12}f_1-a_{22}f_2), \dots$$
form a basis of $\F^n$. Using this, we transform $\bfB$ into a matrix $\bfA=(a_{ij})$ which has the desired form.
\end{proof}

{In the next lemma we will prove that every matrix over a field with at least three elements can be decomposed as a sum of two matrices with prescribed characteristic polynomials. There are several other results regarding such decompositions. We refer the reader to \cite{BM} and \cite{DGG} for recent developments in this area.    

\begin{lemma}\label{lem:desc-princ-2}
	Let $\F\neq \F_2$ be a field, and $\bfA\in\CM_n(\F)$ a non-scalar matrix. Suppose that $P\in \F[X]$ is a polynomial of degree $n-1$ and $Q\in \F[X]$ is of degree $n$. If $t=\Tr(\bfA)-\Tr(P)-\Tr(Q),$ there exist two (non-derogatory) matrices $\bfD\in \CM_{n-1}(\F)$ and $\bfE\in \CM_{n}(\F)$ with the minimal polynomials $P$ and $Q$, respectively, such that $\bfA$ is similar to a matrix of the form
	\begin{equation}\label{eq:matrix-dec}
    \left(\begin{array}{cc}
		\bfD &  \begin{array}{c}
			d_0  \\
			\vdots \\
			d_{n-1}
		\end{array} \\
		\begin{array}{ccc}
			0 & \dots & 0     
		\end{array}   & t 
	\end{array}\right)
	+\bfE.\end{equation}
\end{lemma}

\begin{proof}
	We can assume that $\bfA$ is in Frobenius canonical form $\diag(\bfC_1,\dots, \bfC_s)$, where $\bfC_i\in \CM_{n_i}(\F)$ are companion matrices, $1\leq n_1\leq \dots\leq n_s$, and $n_s>1$. 
	
	Since $\F\neq \F_2$, it follows from Lemma \ref{lem:non-derogatory} that there exists an unreduced upper Hessenberg matrix $\bfD=(d_{ij})$ such that it is similar to $\bfC_P$, and $d_{(j+1)j}- 1_{\F}\neq 0$ for all $j\leq n-1$. Since $n_s>1$, the coefficients of matrix $\bfA-\Diag(\bfD,t)$ satisfy the condition (ii) from Lemma \ref{lem:non-derogatory}. Therefore, there exists a non-derogatory matrix $\bfE\in \CM_n(\F)$ such that it is similar to $\bfC_Q$, and its first $(n-1)$ columns coincide with the corresponding columns of $\bfA-\Diag(\bfD,t)$. 
Since $\bfA$ is equal to the sum described in \eqref{eq:matrix-dec}, the proof is complete.       
\end{proof}

\begin{rem}
We note that the only restriction, other than the degrees, on the selection of the polynomials $P$, $Q$, and the element $t$ required to obtain, up to a similarity, a decomposition as \eqref{eq:matrix-dec} is $t=\Tr(\bfA)-\Tr(P)-\Tr(Q).$    \end{rem}

{ 
\begin{lemma}\label{lem:diag(C^k,t^k)}
Let $\bfB\in \CM_{n-1}(\F)$ and $k$ a positive integer. If $b\neq 0$ is an element of $\F$ such that $b^k$ is not an eigenvalue for $\bfB^k$ then every matrix of the form
$$\left(\begin{array}{cc}
  \bfB^k &  \begin{array}{c}
         d_0  \\
         \vdots \\
         d_{n-1}
    \end{array} \\
  \begin{array}{ccc}
      0 & \dots & 0     
  \end{array}   & b^k 
\end{array}\right)$$ 
is a $k$-th power.
\end{lemma}

\begin{proof}
We consider the equation $$\left(\begin{array}{cc}
  \bfB &  \begin{array}{c}
         x_0  \\
         \vdots \\
         x_{n-1}
    \end{array} \\
  \begin{array}{ccc}
      0 & \dots & 0     
  \end{array}   & b 
\end{array}\right)^k=\left(\begin{array}{cc}
  \bfB^k &  \begin{array}{c}
         d_0  \\
         \vdots \\
         d_{n-1}
    \end{array} \\
  \begin{array}{ccc}
      0 & \dots & 0     
  \end{array}   & b^k 
\end{array}\right).$$
Since $b^k$ is not an eigenvalue for $\bfB^k$, the matrix 
$\bfB^{k-1}+b\bfB^{k-2}+\dots +b^{k-2}\bfB+b^{k-1}I_{n-1}$ is invertible, hence the equation has a solution. 
\end{proof}

}

The following lemma can be deduced from \cite[Corollary 1]{Da05}. Here, $\mathrm{GL}_n(\F_q)$ denotes the general linear group of degree $n$ over $\F_q$.

\begin{lemma}\label{lem:Da05}
If $P\in \F_q[X]$ is an irreducible polynomial of degree $n$ then $\ord_{\mathrm{GL}_n(\F_q)}\bfC_P$ divides $q^n-1$.    
\end{lemma}
}

We are ready to prove the main result of this note.

\begin{thm}\label{thm:nonscalar}
If $n$ and $k$ are positive integers which satisfy the condition \eqref{eq:d<general} then every non-scalar $n\times n$ matrix over $\F_q$ is a sum of two $k$-powers.
\end{thm}

\begin{proof}
Let $\bfA\in \CM_n(\F_q)$ be a non-scalar matrix.

From \eqref{eq:d<general} it follows that $k<q^{(n-1)/2}+1$, hence we can apply Lemma \ref{cor:power-k-companion} to construct a $k$-th power matrix $\bfD\in \CM_{n-1}(\F_q)$ whose characteristic polynomial is irreducible. We observe that we can take this matrix to be of the same form as in the proof of Lemma \ref{lem:desc-princ-2}. Moreover, we consider an element $t\in \F_q$ whics is $k$-th power. 

Since $k\geq 1$, the condition \eqref{eq:d<general} implies $n\geq 3$, hence $t$ is not an eigenvalue for $\bfD$. It follows from Lemma \ref{lem:diag(C^k,t^k)} that the matrices of the form $$\left(\begin{array}{cc}
  \bfD &  \begin{array}{c}
         x_0  \\
         \vdots \\
         x_{n-1}
    \end{array} \\
  \begin{array}{ccc}
      0 & \dots & 0     
  \end{array}   & t 
\end{array}\right)$$ are $k$-th powers.

We write $k=p^a k'$, with $\gcd(k',q)=1$. Since $n$ and $k'$ satisfy the condition \eqref{eq:d<general}, we can apply Proposition \ref{prop:k-power-polynomial} to observe that there exists an irreducible polynomial $Q$ such that $\Tr(Q)=\Tr(\bfA)-\Tr(\bfD)-t$ and the associated companion matrix $\bfC_Q$ is a $k'$-th power. If $\bfC_Q=(\bfE')^{k'}$, it  follows from Lemma \ref{lem:Da05} that the order of $\bfE'$ in $\mathrm{GL}_n(\F_q)$ is coprime with $p$. Then there exists a matrix $\bfE$ such that $\bfE'=\bfE^{p^a}$, hence $\bfC_Q=\bfE^k$. Since every matrix $\bfE\in \CM_n(\F_q)$ with the characteristic polynomial $Q$ is similar to $\bfC_Q$, it follows that we can apply Lemma \ref{lem:desc-princ-2} to obtain the conclusion.    
\end{proof}
}

\begin{rem}
 a) If $n$ is a prime, the condition \eqref{eq:d<general} can be written as 
		\begin{equation}
			\label{eq:d-prim<general}k<\frac{q^{n/2}-1}{q-1} \end{equation}

 b) If $n\geq 4$ then the condition \eqref{eq:d<general} can be replaced with
		\begin{equation}
			\label{eq:d-simplu}
			k\leq q^{\frac{n}{2}-1}.
		\end{equation}	 
In order to see that the condition \eqref{eq:d-simplu} implies \eqref{eq:d<general}, it is enough to consider the following cases: 
		
		\noindent \textbf{case 1: $\omega(n)=1$.} In this case, we observe that it is enough to have 
		$$k\leq \frac{q^{n/2}-1}{q-1}.$$
		
		\noindent \textbf{case 2: $\omega(n)=2$.} The condition \eqref{eq:d<general} becomes  
		$$k\leq \frac{q^{n/2}-q-1}{q-1},$$ which is satisfied if \eqref{eq:d-simplu} is true.

		\noindent \textbf{case 3: $\omega(n)\geq 3$.} In this case $n\geq 30$. Since $\omega(n)<\log_2(n)$, it is easy to observe that the condition \eqref{eq:d-simplu} is enough in order to obtain the conclusion. 
\end{rem}

In the following, we will include the scalar matrices in our results. Let us remind that $\F_{q^n}$ can be viewed as an $\F_q$-algebra $$\F_q[\bfC]=\{\alpha_0+\alpha_1\bfC+\dots+\alpha_{n-1}\bfC^{n-1}\mid \alpha_0,\dots,\alpha_{n-1}\in \F_q\},$$ 
where $\bfC\in \CM_n(\F_q)$ is a matrix whose characteristic polynomial is irreducible. Therefore, in order to decompose a non-zero scalar matrix from $\CM_n(\F_q)$ as a sum of two $k$-th powers, it is enough to construct such a decomposition in $\F_q[\bfC]$. 

\begin{lemma}\label{lem:scalar}
	Let $\F_q$ be a finite field of cardinality $q$. If $k$ and $n$ are positive integers and $q^n> (k-1)^4$ then 
 every $n\times n$ scalar matrix is a sum of two $k$-th powers.
\end{lemma}

\begin{proof}
	 From \cite[Corollary 6.12 (a)]{Sm-book} it follows that every non-zero element of $\F_q[\bfC]$ is a sum of two $k$-th powers. In particular, every scalar $n\times n$ matrix is a sum of two $k$-th powers.  
\end{proof}

\begin{rem}
In fact, the bound $q^n> (k-1)^4$ can be replaced by a weaker condition, see \cite[Corollary 6.12]{Sm-book} or \cite[Example 6.38]{lidl}. 
\end{rem}

Now, using Theorem \ref{thm:nonscalar} and Lemma \ref{lem:scalar}, we obtain the following

\begin{cor}\label{cor:main}
Let $F_q$ be a finite field of cardinality $q$. If $n$ and $q$ are positive integers such that $q^n> (k-1)^4$ then every $n\times n$ matrix is a sum of two $k$-th powers. 
\end{cor}

The following corollary can be of independent interest.

\begin{cor}\label{cor:main<q}
If $\F_q$ is a finite field of cardinality $q$ and $n\geq 4$ is an integer, then for every $k\leq q$, every $n\times n$ matrix is a sum of two $k$-th powers. 	
\end{cor}

}

\subsection*{Acknowledgements} I would like to thank George C\u at\u alin \c Turca\c s for fruitful discussions about this subject.

This work was supported by a grant of the Ministry of Research, Innovation and Digitization, CNCS - UEFISCDI, project number PN-IV-P1-PCE-2023-0060, within PNCDI IV.


\end{document}